\documentclass[12pt]{article}
\usepackage[margin=1.3in]{geometry}

\usepackage{amsfonts,amsmath,amssymb,amsthm}
\usepackage{CJK,CJKnumb,CJKulem,cases,cite,color,curves}
\usepackage{bm,dsfont,enumitem,extarrows,fontenc,graphicx,ifthen,latexsym,listings,makeidx,mathrsfs,psfrag,times,xcolor}
\usepackage[colorlinks,citecolor=green,linkcolor=red]{hyperref}
\usepackage[title]{appendix}

\let\oldbibliography\thebibliography
\renewcommand{\thebibliography}[1]{\oldbibliography{#1}\setlength{\itemsep}{0pt}}

\numberwithin{equation}{section}

\makeindex
\newtheorem{theorem}{Theorem}[section]

\newtheorem{proposition}[theorem]{Proposition}

\newtheorem{corollary}[theorem]{Corollary}

\newcommand{\R}{\mathbb R}

\usepackage[title]{appendix}

\begin{document}

\title{\textbf{Existence of solutions to a conformally invariant integral equation involving Poisson-type kernels}\bigskip}

\author{\medskip Xusheng Du, \quad Tianling Jin\footnote{T. Jin was partially supported by Hong Kong RGC grant GRF 16306918 and NSFC grant 12122120.}, \quad Hui Yang}

\date{\today}

\maketitle

\begin{abstract}

In this paper, we study existence of solutions to a conformally invariant integral equation involving Poisson-type kernels. Such integral equation has a stronger non-local feature and is not the dual of any PDE. We obtain the existence of solutions in the antipodal symmetry class.

\medskip


\medskip


\end{abstract}

\section{Introduction}

In \cite{HWY08}, Hang-Wang-Yan established the following sharp integral inequality: 
\begin{equation}\label{In-01}
\| v \|_{ L^\frac{2 n}{n - 2} (B_1) } \leq n^{ - \frac{n - 2}{2 (n - 1)} } \omega_n^{ - \frac{n - 2}{2 n (n - 1)} } \| v \|_{ L^\frac{2 (n - 1)}{n - 2} (\partial B_1) }
\end{equation}
for every harmonic function $v$ on the unit ball $B_1 \subset \R^n$ ($n \geq 3$), where $\omega_n$ is the Euclidean volume of $B_1$. They also classified all the maximizers by showing that the equality holds if and only if $v = \pm 1$ up to a conformal transform on the unit sphere $\partial B_1$. This is actually a higher dimensional generalization of Carleman's inequality \cite{C}, which was used by Carleman to prove the classical isoperimetric inequality. Let $g_{\R^n}$ be the Euclidean metric on $\R^n$. Then for a positive harmonic function $v$ on $B_1$, the scalar curvature of $g = v^\frac{4}{n - 2} g_{\R^n}$ on $B_1$ is identically zero. Moreover, under the metric $g$, the volume of $B_1$ and the area of $\partial B_1$ are equal to $\int_{B_1} v^\frac{2 n}{n - 2} d \xi$ and $\int_{\partial B_1} v^\frac{2 (n - 1)}{n - 2} d s$, respectively. Hence, the inequality \eqref{In-01} can be considered as an isoperimetric inequality in the conformal class of $g_{\R^n}$ for which the scalar curvature vanishes. In \cite{HWY09}, Hang-Wang-Yan further obtained a generalization of \eqref{In-01} on a smooth compact Riemannian manifold of dimension $n \geq 3$ with non-empty boundary by introducing an isoperimetric ratio over the scalar-flat conformal class. It was conjectured there that unless the manifold is conformally diffeomorphic to the Euclidean ball, the supremum of the isoperimetric ratio over the scalar-flat conformal class is always strictly larger than that in the Euclidean ball, so that the maximizers would exist. This conjecture was confirmed in higher dimensions under certain geometric assumptions by Jin-Xiong \cite{JX18} and Chen-Jin-Ruan \cite{CJR19}, and also was confirmed for balls with a small hole by Gluck-Zhu \cite{GZ}.

Using the M{\" o}bius transformation in \eqref{Mob}, the equivalent form of \eqref{In-01} in the upper half-space is given by
\begin{equation}\label{Upp}
\| \mathcal{P} u \|_{ L^\frac{2 n}{n - 2} (\R^n_+) } \leq n^{ - \frac{n - 2}{2 (n - 1)} } \omega_n^{ - \frac{n - 2}{2 n (n - 1)} } \| u \|_{ L^\frac{2 (n - 1)}{n - 2} (\R^{n - 1}) },
\end{equation}
where $\R^{n - 1}$ is the boundary of $\R^n_+$ and $\mathcal{P} u$ is the Poisson integral of $u$ in the upper half-space. The maximizers are $u(y') = c (\lambda^2 + |y' - y'_0|^2)^{ - \frac{n - 2}{2} }$ for some constant $c$, positive constant $\lambda$, and $y'_0 \in \R^{n - 1}$. In \cite{Chen14}, Chen proved an analogous inequality for a one-parameter family $\{ \mathcal{P}_a \}_{2 - n < a < 1}$ of Poisson-type kernels in $\R^n_+$. More specifically, let the parameter $a$ satisfy $2 - n < a < 1$ with $n \geq 2$, and define the Poisson-type kernels
$$
P_a (y', x) = c_{n, a} \frac{ x_n^{1 - a} }{ (|x' - y'|^2 + x_n^2)^\frac{n - a}{2} } ~~~~~~ \textmd{for} ~ y' \in \R^{n - 1}, ~ x \in \R_+^n,
$$
where $x = (x', x_n) \in \R_+^n = \R^{n - 1} \times (0, + \infty)$ and $c_{n, a}$ is the positive normalization constant such that $\int_{ \R^{n - 1} } P_a (y', x) d y' = 1$. Consider the following Poisson-type integral
\begin{equation}\label{eq:poissontypeR}
(\mathcal{P}_a u) (x) = \int_{ \R^{n - 1} } P_a (y', x) u(y') d y' ~~~~~~ \textmd{for} ~ x \in \R_+^n.
\end{equation}
It becomes the Poisson integral when $a = 0$ (i.e., $\mathcal{P}_0 = \mathcal{P}$). Chen \cite{Chen14} proved the following sharp integral inequality
\begin{equation}\label{Pa}
\| \mathcal{P}_a u \|_{ L^\frac{2 n}{n + a - 2} (\R^n_+) } \leq \mathcal{S}_{n, a} \| u \|_{ L^\frac{2 (n - 1)}{n + a - 2} (\R^{n - 1}) },
\end{equation}
where the sharp constant $\mathcal{S}_{n, a}$ depends only on $n$ and $a$. This Poisson-type integral \eqref{eq:poissontypeR} was used earlier by Caffarelli-Silvestre \cite{CS07} to localize the fractional Laplacian operator. Indeed, when $-1<a<1$, then it was shown in \cite{CS07} that 
\begin{equation*}
\begin{split}
\mbox{div} [x_n^{a}\nabla (\mathcal{P}_a u)]&=0\quad\mbox{in }\R^n_+,\\
-\lim_{x_n\to 0^+} x_n^{a}\partial_{x_n} (\mathcal{P}_a u) &= C_{n,a} (-\Delta)^{\frac{1-a}{2}} u\quad\mbox{on }\R^{n-1},
\end{split}
\end{equation*}
where $C_{n,a}$ is a positive constant and $(-\Delta)^{\frac{1-a}{2}} $ is the fractional Laplacian operator. See also Yang \cite{Yang} for higher order extensions for the fractional Laplacian. We refer to Dou-Guo-Zhu \cite{DGZ}, Gluck \cite{G} and the references therein for other related integral inequalities.

One can  define the Poisson-type integral $\widetilde{ \mathcal{P} }_a v$ on $B_1$ as the pull back operator of $\mathcal{P}_a$ via the M{\" o}bius transformation:
\begin{equation}\label{Mob}
F: \  \overline{\R_+^n} \to \overline{B}_1, ~~~~~~ x \mapsto \frac{2 (x + e_n)}{|x + e_n|^2} - e_n,
\end{equation} 
where $e_n = (0, \dots, 0, 1) \in \R^n$. Then for $y'\in \R^{n-1}$,
$$
F(y', 0) = \bigg( \frac{2 y'}{1 + |y'|^2}, \frac{1 - |y'|^2 }{1 + |y'|^2} \bigg) \in \partial B_1
$$
is the inverse of the stereographic projection. For $v \in L^\frac{2 (n - 1)}{n + a - 2} (\partial B_1)$, let
$$
u(y') = \bigg( \frac{ \sqrt{2} }{|(y',0) + e_n|} \bigg)^{n + a - 2} v( F(y', 0) ), 
$$
and define
$$
 (\widetilde{ \mathcal{P} }_a v) ( F(x) )=\bigg( \frac{|x + e_n|}{ \sqrt{2} } \bigg)^{n + a - 2} (\mathcal{P}_a u) (x).
$$
That is, 
$$
(\widetilde{ \mathcal{P} }_a v) \circ F(x) = |x + e_n|^{n + a - 2} \mathcal{P}_a \bigg( \frac{ v \circ F(y', 0) }{ |(y',0) + e_n|^{n + a - 2} } \bigg) ~~~~~~ \textmd{for} ~ v \in L^\frac{2 (n - 1)}{n - 2 + a} (\partial B_1).
$$
By a direct calculation, for $v \in L^\frac{2 (n - 1)}{n + a - 2} (\partial B_1)$, the Poisson-type integral $\widetilde{\mathcal{P}}_a v$ on the unit ball has the following explicit form:
\begin{equation}\label{eq:poissontypeS}
(\widetilde{ \mathcal{P} }_a v) (\xi) = \int_{\partial B_1} \widetilde P_a (\eta, \xi) v(\eta) d s_\eta ~~~~~~ \textmd{for} ~ \xi \in B_1,
\end{equation}
where
$$
\widetilde P_a (\eta, \xi) = 2^{a - 1} c_{n, a} \frac{ (1 - |\xi|^2)^{1 - a} }{ |\xi - \eta|^{n - a} }.
$$
Then, it follows from \eqref{Pa} that we have the following sharp inequality
\begin{equation}\label{eq:sharp}
\| \widetilde{ \mathcal{P} }_a v \|_{ L^\frac{2 n}{n + a - 2} (B_1) } \leq \mathcal{S}_{n, a} \| v \|_{ L^\frac{2 (n - 1)}{n + a - 2} (\partial B_1) }.
\end{equation}

From now on, for simplicity, we will use the unified notation $\mathcal{P}_a v$ to denote either the Poisson-type integral \eqref{eq:poissontypeR} of $v$ on the upper half space or the Poisson-type integral \eqref{eq:poissontypeS} of $v$ on the unit ball, whenever there is no confusion. Inspired by Hang-Wang-Yan \cite{HWY09} on the proof of inequality \eqref{In-01}, for a positive function $K \in C^1 (\partial B_1)$ we consider the weighted isoperimetric ratio
$$
I(v, K) = \frac{ \int_{B_1} |\mathcal{P}_a v|^\frac{2 n}{n + a - 2} d \xi }{ \Big( \int_{\partial B_1} K |v|^\frac{2 (n - 1)}{n + a - 2} d s \Big)^\frac{n}{n - 1} } ~~~~~~ \textmd{for} ~ v \in L^\frac{2 (n - 1)}{n + a - 2} (\partial B_1).
$$
In this paper, motivated by the classical Nirenberg problem we would like to study existence of positive solutions to the Euler-Lagrange equation of the functional $I(v, K)$ for a given function $K > 0$. The Euler-Lagrange equation can be written as the following integral equation
\begin{equation}\label{eq:integralequation}
K(\eta) v(\eta)^\frac{n - a}{n + a - 2} = \int_{B_1} P_a (\eta, \xi) \left[(\mathcal{P}_a v) (\xi)\right]^\frac{n - a + 2}{n + a - 2} d \xi, ~~~~~~ v > 0 ~~~~~~ \textmd{on} ~ \partial B_1.
\end{equation}
This equation is critical and conformally invariant. Moreover, it is not always solvable by a Kazdan-Warner type obstruction (see Lemma 3.1 of Hang-Wang-Yan \cite{HWY09} for $a=0$). In this paper, we show the following existence result.
\begin{theorem}\label{thm:existence} Suppose that $n \geq 2$ and $2 - n < a < 1$. Let $K \in C^1 (\partial B_1)$ be a positive function satisfying $K(\xi) = K(- \xi)$ for every $\xi\in\partial B_1$. If
\begin{equation}\label{eq:maxmin}
\frac{ \max_{\partial B_1} K }{ \min_{\partial B_1} K } < 2^\frac{1}{n},
\end{equation}
then equation \eqref{eq:integralequation} has at least one positive H\"older continuous solution.
\end{theorem}

The existence of solutions to the Nirenberg problem for prescribed antipodal symmetric functions was established by Moser \cite{Moser} in dimension two, and by Escobar-Schoen \cite{ES} in higher dimensions under a  flatness assumption near the prescribed function's maximum point. For the generalized Nirenberg problem for $Q$-curvature and fractional $Q$-curvatures, similar results have been obtained by Robert \cite{Robert} and Jin-Li-Xiong \cite{JLX15, JLX17}, respectively. In the case $a = 0$, the existence of solutions to \eqref{eq:integralequation} with antipodal symmetric functions $K$ has been proved by Xiong \cite{X}  under a global flatness condition at $K$'s minimum point. Our condition is slightly weaker, although it is still a (not arbitrarily small, though) perturbation result. We do not know whether a local flatness condition would be sufficient. The difficulty is that the antipodal symmetry does not provide a desirable positive mass in our setting, which is different from the Nirenberg problem or the Yamabe problem.  Note that equation \eqref{eq:integralequation} has a stronger non-local feature and is not the dual of any PDE. This, as already shown in \cite{X},  will lead to some differences from the classical Nirenberg problem \cite{JLX17}.

This paper is organized as follows. In Section \ref{sec:Preliminaries}, we collect some elementary properties of the Poisson extension as a preparation. In Section \ref{sec:blowup}, we show the blow up procedure for the non-linear integral equation \eqref{eq:integralequation}. In Section \ref{sec:variational}, we use a variational method to prove Theorem \ref{thm:existence}.

\section{Preliminaries}\label{sec:Preliminaries}

From now on, we denote $x = (x', x_n) \in \R^{n - 1} \times \R$ as the point in $\R^n$, $B_R (x)$ as the open ball of $\R^n$ with radius $R$ and center $x$, $B_R^+(x)$ as $B_R(x) \cap \mathbb{R}_+^n$, and $B'_R (x')$ as the open ball in $\R^{n - 1}$ with radius $R$ and center $x'$. For simplicity, we also write $B_R(0)$, $B_R^+(0)$ and $B'_R (0)$ as $B_R$, $B_R^+$ and $B'_R$, respectively.

Here we list several properties of the Poisson-type extension operator $\mathcal{P}_a$.

\begin{proposition}\label{lem:compactembeddingR} Suppose that $n \geq 2$ and $2 - n < a < 1$. If $1 \leq p < \infty$ and $1 \leq q < \frac{n p}{n - 1}$, then the operator
$$
\mathcal{P}_a : L^p (\R^{n - 1}) \to L_{loc}^q (\overline{\R_+^n})
$$
is compact.
\end{proposition}

\begin{proof} 
The proof is the same as that of \cite[Corollary 2.2]{HWY08}.
\end{proof}

\begin{corollary}\label{cor:compactembeddingS} Suppose that $n \geq 2$ and $2 - n < a < 1$. If $1 \leq p < \infty$ and $1 \leq q < \frac{n p}{n - 1}$, then the operator
$$
\mathcal{P}_a : L^p (\partial B_1) \to L^q (B_1)
$$
is compact.
\end{corollary}
\begin{proof} 
The proof is the same as that of \cite[Corollary 2.1]{HWY09}.
\end{proof}

In order to establish regularity, we need the following simple fact 
\begin{equation}\label{eq:DkPa}
|\nabla_{x'}^k P_a (y', x)| = |\nabla_{y'}^k P_a (y', x)| \leq C(n, a, k) x_n^{1 - a} (|x' - y'|^2 + x_n^2)^{ - \frac{n - a + k}{2} }
\end{equation}
for $x' \neq y'$ and $k \geq 1$. 

\begin{theorem}\label{thm:Holder} Suppose that $n \geq 2$, $2 - n < a < 1$ and $\frac{2 (n - 1)}{n + a - 2} \leq p < \infty$. Let $K \in C^1 (\R^{n - 1})$ be a positive function. If $u \in L_{loc}^p (\R^{n - 1})$ is non-negative, not identically zero and satisfies
\begin{equation}\label{eq:Holder}
K (y') u(y')^{p - 1} = \int_{\R_+^n} P_a (y', x) \left[(\mathcal{P}_a u) (x)\right]^\frac{n - a + 2}{n + a - 2} d x,
\end{equation}
then $u \in C_{loc}^\beta (\R^{n - 1})$ for any $ \beta \in (0, 1)$. 
\end{theorem}

The proof of this H\"older regularity is given in the appendix.

Finally, we also need the following Liouville theorem proved by Wang-Zhu \cite{WZ20}.

\begin{theorem}[Wang-Zhu \cite{WZ20}] \label{Wang-Zhu} 
Suppose that $n \geq 2$ and $a < 1$. If $\Phi \in C^2 (\R_+^n) \cap C^0 (\overline{\R_+^n})$ is a solution of
$$
\left\{
\aligned
- \, {\rm div} (x_n^{a} \nabla \Phi) & = 0 ~~~~~~ \textmd{in} ~ \R_+^n, \\
\Phi & = 0 ~~~~~~ \textmd{on} ~ \R^{n - 1}
\endaligned
\right.
$$
and is bounded from below in $\R_+^n$. Then
$$
\Phi (x) = C x_n^{1 - a}
$$
for some constant $C \geq 0$.
\end{theorem}

\section{A blow-up analysis}\label{sec:blowup}

The local blow up analysis for the non-local integral equation \eqref{eq:Holder} is as follows.

\begin{theorem}\label{thm:blowup} Suppose that $n \geq 2$ and $2 - n < a < 1$. Let $\frac{2 (n - 1)}{n + a - 2} \leq p_i < \frac{2 n}{n + a - 2}$ be a sequence of numbers with $\lim_{i \to \infty} p_i = \frac{2 (n - 1)}{n + a - 2}$, and $K_i \in C^1 (B_1')$ be a sequence of positive functions satisfying
$$
K_i \geq \frac{1}{c_0}, ~~~~~~ \| K_i \|_{C^1 (B_1')} \leq c_0
$$
for some constant $c_0 \geq 1$ independent of $i$. Suppose that $u_i \in C(\R^{n - 1})$ is a sequence of non-negative solutions of
\begin{equation}\label{Kiuipi}
K_i (y') u_i (y')^{p_i - 1} = \int_{\R_+^n} P_a (y', x) \left[(\mathcal{P}_a u_i) (x)\right]^\frac{n - a + 2}{n + a - 2} d x ~~~~~~ \textmd{for} ~ y' \in B_1'
\end{equation}
and $u_i (0) \to + \infty$ as $i \to \infty$. Suppose that $R_i u_i (0)^{ p_i - \frac{2 n}{n + a - 2} } \to 0$ for some $R_i \to + \infty$ and
\begin{equation}\label{eq:uixlebui0}
u_i (y') \leq b u_i (0) ~~~~~~ \textmd{for} ~ |y'| < R_i u_i (0)^{ p_i - \frac{2 n}{n + a - 2} },
\end{equation}
where $b > 0$ is independent of $i$. Then, after passing to a subsequence, we have
\begin{equation}\label{eq:phiiconverge}
\phi_i (y') : = \frac{1}{u_i (0)} u_i \big( u_i (0)^{ p_i - \frac{2 n}{n + a - 2} } y' \big) \to \phi (y') ~~~~~~ \textmd{in} ~ C_{loc}^{1/2} (\R^{n - 1}),
\end{equation}
where $\phi > 0$ satisfies
$$
K \phi (y')^\frac{n - a}{n + a - 2} = \int_{\R_+^n} P_a (y', x) \left[(\mathcal{P}_a \phi) (x)\right]^\frac{n - a + 2}{n + a - 2} d x ~~~~~~ \textmd{for} ~ y' \in \R^{n - 1}
$$
and $K : = \lim_{i \to \infty} K_i (0) > 0$ along the subsequence.
\end{theorem}

\begin{proof} It follows from \eqref{Kiuipi} and \eqref{eq:phiiconverge} that $\phi_i$ satisfies the equation
\begin{equation}\label{eq:Hiphii}
H_i (y') \phi_i (y')^{p_i - 1} = \int_{\R_+^n} P_a (y', x) \left[(\mathcal{P}_a \phi_i) (x)\right]^\frac{n - a + 2}{n + a - 2} d x ~~~~~~ \textmd{for} ~ |y'| < R_i,
\end{equation}
where $H_i (y') : = K_i \big( u_i (0)^{ p_i - \frac{2 n}{n + a - 2} } y' \big)$. Moreover, by \eqref{eq:uixlebui0}, we have
\begin{equation}\label{eq:phiileb}
0 \leq \phi_i (y') \leq b ~~~~~~ \textmd{for} ~ |y'| < R_i.
\end{equation}
The proof consists of two steps.

\vskip0.1in

{\bf Step 1. Estimate the locally uniform bound of $\{ \phi_i \}$ in some H{\" o}lder spaces.}

\vskip0.1in

Fixing $100 < R < R_i/2$ for large $i$, we can define
$$
\Phi_i' = \mathcal{P}_a (\chi_{B_R'} \phi_i) ~~~~~~ \textmd{and} ~~~~~~ \Phi_i'' = \mathcal{P}_a ( (1 - \chi_{B_R'}) \phi_i),
$$
where $\chi_{B_R'}$ is the characterization function of $B_R'$. Then
$$
\mathcal{P}_a \phi_i = \Phi_i' + \Phi_i''.
$$
By \eqref{eq:phiileb} and the property of $\mathcal{P}_a$ we can get
\begin{equation}\label{eq:Phii1leb}
0 \leq \Phi_i' \leq b.
\end{equation}
Since $K_i \leq c_0$ on $B'_1$, by \eqref{eq:Hiphii} and \eqref{eq:phiileb}, for any $|y'| < R - 2$ we have,
$$
\aligned
c_0 b^{p_i - 1}
& \geq \int_{ B_{1/2} (y', 1) } P_a (y', x) \left[(\mathcal{P}_a \phi_i) (x)\right]^\frac{n - a + 2}{n + a - 2} d x \\
& \geq \frac{1}{C} \int_{ B_{1/2} (y', 1) } \left[(\mathcal{P}_a \phi_i) (x)\right]^\frac{n - a + 2}{n + a - 2} d x \\
& \geq \frac{1}{C} \left[(\mathcal{P}_a \phi_i) (\bar{x})\right]^\frac{n - a + 2}{n + a - 2}
\endaligned
$$
for some $\bar{x} \in \overline{B}_{1/2} (y', 1)$, where we used the mean value theorem for integrals in the last inequality and $C > 0$ depends only on $n$ and $a$. It follows that
$$
\Phi_i'' (\bar{x}) \leq (\mathcal{P}_a \phi_i) (\bar{x}) \leq C b^\frac{(p_i - 1) (n + a - 2)}{n - a + 2}.
$$
Since $|\bar{x}'| \leq R - 1$ and $\frac{1}{2} \leq \bar{x}_n \leq \frac{3}{2}$,
$$
\aligned
C b^\frac{(p_i - 1) (n + a - 2)}{n - a + 2} \geq \Phi_i'' (\bar{x}) & = c_{n, a} \int_{ \R^{n - 1} \setminus B_R' } \frac{ \bar{x}_n^{1 - a} }{ (|\bar{x}' - z'|^2 + \bar{x}_n^2)^\frac{n - a}{2} } \phi_i (z') d z' \\
& \geq \frac{1}{C} \int_{ \R^{n - 1} \setminus B_R' } \frac{\phi_i (z')}{ |\bar{x}' - z'|^{n - a} } d z'. 
\endaligned
$$
Therefore, for any $|y'| < R - 2$ and $x \in B_1' (y') \times (0, 1]$, we have
\begin{equation}\label{Phi02698}
\aligned
 \frac{\Phi_i'' (x)}{x_n^{1 - a}} &  \leq C \int_{ \R^{n - 1} \setminus B_R' } \frac{\phi_i (z')}{ |x' - z'|^{n - a} } d z' \\ 
& \leq C \int_{ \R^{n - 1} \setminus B_R' } \frac{\phi_i (z')}{ |\bar{x}' - z'|^{n - a} } d z' \\
& \leq C b^\frac{(p_i - 1) (n + a - 2)}{n - a + 2},
\endaligned
\end{equation}
where the second inequality holds since
$$
|\bar{x}' - z'| \leq |\bar{x}' - x'| + |x' - z'| \leq 2 + |x' - z'| \leq 3 |x' - z'|.
$$
This together with \eqref{eq:Phii1leb} implies that
$$
(\mathcal{P}_a \phi_i) (x) \leq C(n, a, c_0, b) ~~~~~~ \forall ~ x \in B_{R - 2}' \times (0, 1].
$$
Using the above estimate,  we have by direct calculations that
$$
\bigg\| \int_{B_{R - 2}' \times (0, 1]} P_a (y', x) \left[(\mathcal{P}_a \phi_i) (x)\right]^\frac{n - a + 2}{n + a - 2} d x \bigg\|_{ C^\beta (B_{R - 3}') } \leq C(n, a, b, c_0, R, \beta)
$$
for any $\beta \in (0, 1)$. On the other hand, for $|y'| < R - 3$, by \eqref{eq:DkPa} we have
$$
\aligned
& \bigg| \nabla_{y'} \bigg( \int_{ \R_+^n \setminus B_{R - 2}' \times (0, 1] } P_a (y', x) \left[(\mathcal{P}_a \phi_i) (x)\right]^\frac{n - a + 2}{n + a - 2} d x \bigg) \bigg| \\
& \leq C \int_{ \R_+^n \setminus B_{R - 2}' \times (0, 1] } P_a (y', x) \left[(\mathcal{P}_a \phi_i) (x)\right]^\frac{n - a + 2}{n + a - 2} d x \\
& \leq C H_i (y') \phi_i (y')^{p_i - 1} \\
& \leq C b^{p_i - 1} \\
& \leq C(n, a, c_0, b).
\endaligned
$$
Combing the above two estimates and using \eqref{eq:Hiphii}, we can obtain
\begin{equation}\label{eq:phiipiHolder}
\| \phi_i^{p_i - 1} \|_{ C^{3/4} (B_{R - 3}') } \leq C(n, a, b, c_0, R).
\end{equation}

Since $\phi_i (0)^{p_i - 1} = 1$, by \eqref{eq:phiipiHolder} there exists $\delta > 0$ depending only on $n$, $a$, $b$ and $c_0$ such that $\phi_i (y')^{p_i - 1} \geq \frac{1}{2}$ for all $|y'| < \delta$. Hence,
$$
(\mathcal{P}_a \phi_i) (x) \geq \frac{1}{C} \int_{B_\delta'} \frac{ x_n^{1 - a} }{ (|x' - y'|^2 + x_n^2)^\frac{n - a}{2} } 2^{ - \frac{1}{p_i - 1} } d y' \geq \frac{1}{C} \frac{ x_n^{1 - a} }{ (1 + |x|)^{n - a} }.
$$
Again, using \eqref{eq:Hiphii} we can get for $|y'| < R - 3$,
$$
\phi_i (y')^{p_i - 1} \geq \frac{1}{C(n, a, c_0, b, R)} > 0.
$$
This together with \eqref{eq:phiipiHolder} implies that
\begin{equation}\label{eq:phiiHolder}
\| \phi_i \|_{ C^{3/4} (B_{R - 3}') } \leq C(n, a, c_0, b, R).
\end{equation}
Hence, \eqref{eq:phiiconverge} is proved.

\vskip0.1in

{\bf Step 2. Show the convergence of $\mathcal{P}_a \phi_i$ and the equation of $\phi_i$.}

\vskip0.1in

Fixing $100 < R < R_i/2$ for large $i$, we write \eqref{eq:Hiphii} as
\begin{equation}\label{eq:phiitwoparts}
H_i (y') \phi_i (y')^{p_i - 1} = \int_{B_R^+} P_a (y', x) \left[(\mathcal{P}_a \phi_i) (x)\right]^\frac{n - a + 2}{n + a - 2} d x + h_i (R, y'),
\end{equation}
where
$$
h_i (R, y') = \int_{\R_+^n \setminus B_R^+} P_a (y', x) \left[(\mathcal{P}_a \phi_i) (x)\right]^\frac{n - a + 2}{n + a - 2} d x \geq 0.
$$
By \eqref{eq:DkPa} and \eqref{eq:Hiphii}, for any $|y'| < R - 1$, we have
$$
|\nabla h_i (R, y')| \leq C h_i (R, y') \leq C H_i (y') \phi_i (y')^{p_i - 1} \leq C(n, a, c_0, b).
$$
Therefore, after passing to a subsequence,
$$
h_i (R, y') \to h(R, y')
$$
for some non-negative function $h \in C^{3/4} (B_{R - 1})$.

Similar as in Step 1, we write $\mathcal{P}_a \phi_i$ into following two parts $\Phi_i'$ and $\Phi_i''$: 
$$
\Phi_i' = \mathcal{P}_a (\eta_R \phi_i) ~~~~~~ \textmd{and} ~~~~~~ \Phi_i'' = \mathcal{P}_a ( (1 - \eta_R) \phi_i),
$$
where $\eta_R$ is a smooth cut-off function satisfying $\eta_R \equiv 1$ in $B_{R-4}'$ and $\eta_R \equiv 0$ in $(B_{R-3}')^c$. By using \eqref{eq:phiiHolder} and noticing that 
$$
\Phi_i' (x) = c_{n,a}\int_{\mathbb{R}^{n-1}} \frac{1}{(|z'|^2 +1)^{\frac{n-a}{2}}} (\eta_R \phi_i)(x' - x_n z') dz', 
$$
we can obtain $\| \Phi_i' \|_{ C^{\alpha} (B_{R/2}^+) } \leq C(n, a, c_0, b, R)$ with $\alpha:= \min\{3/4, 1-a\}  > 0$.  On the other hand, similar to \eqref{Phi02698} we have  
$$
\left\| \frac{\Phi_i'' }{x_n^{1 - a}} \right\|_{C^1(B_{R/2}^+)} \leq C(n, a, c_0, b, R), 
$$
and hence $\| \Phi_i'' \|_{ C^{\alpha} (B_{R/2}^+) } \leq C(n, a, c_0, b, R)$. Therefore, after passing to a subsequence, we have
$$
\mathcal{P}_a \phi_i \to \tilde{\Phi} ~~~~~~ \textmd{in} ~ C_{loc}^{\alpha/2} (\overline{\R_+^n})
$$
for some $\tilde{\Phi} \geq 0$ satisfying 
$$
\left\{
\aligned
-\,{\rm div} (x_n^{a} \nabla \tilde{\Phi}) & = 0 ~~~~~~ \textmd{in} ~ \R_+^n, \\
\tilde{\Phi} & = \phi ~~~~~~ \textmd{on} ~ \R^{n - 1}.
\endaligned
\right.
$$
From \eqref{eq:phiileb}, we know that $0 \leq \phi \leq b$ in the whole $\R^{n - 1}$, and thus $\mathcal{P}_a \phi$ is bounded in $\R_+^n$. Hence, $\tilde{\Phi} - \mathcal{P}_a \phi \in C^2 (\R_+^n) \cap C^0 (\overline{\R_+^n})$ satisfies
$$
\left\{
\aligned
-\,{\rm div} (x_n^{a} \nabla (\tilde{\Phi} - \mathcal{P}_a \phi) ) & = 0 ~~~~~~ \textmd{in} ~ \R_+^n, \\
\tilde{\Phi} - \mathcal{P}_a \phi & = 0 ~~~~~~ \textmd{on} ~ \R^{n - 1}.
\endaligned
\right.
$$
It follows from the Liouville-type result in Theorem \ref{Wang-Zhu} that
\begin{equation}\label{eq:tildePhi}
\tilde{\Phi} = \mathcal{P}_a \phi + c_1 x_n^{1 - a}
\end{equation}
for some constant $c_1 \geq 0$. Sending $i \to \infty$ in \eqref{eq:phiitwoparts}, we have
\begin{equation}\label{eq:phtwoparts}
K \phi (y')^\frac{n - a}{n + a - 2} = \int_{B_R^+} P_a (y', x) \tilde{\Phi} (x)^\frac{n - a + 2}{n + a - 2} d x + h(R, y').
\end{equation}
If $c_1 > 0$ in \eqref{eq:tildePhi}, taking $y' = 0$ and sending $R \to \infty$ we obtain that 
$$
K \phi (0)^\frac{n - a}{n + a - 2} \geq \int_{B_R^+} P_a (0, x) \tilde{\Phi} (x)^\frac{n - a + 2}{n + a - 2} d x \to \infty.
$$
This is a contradiction. Hence, $c_1 = 0$ and $\tilde{\Phi} = \mathcal{P}_a \phi$.

Now we adapt some arguments in \cite[Proposition 2.9]{JLX17}. By \eqref{eq:phtwoparts}, $h(R, y')$ is non-increasing with respect to $R$. Notice that for $R \gg |y'|$,
$$
\aligned
\frac{ R^{n - a} }{ (R + |y'|)^{n - a} } h_i (R, 0) & \leq h_i (R, y') \\
&= c_{n, a} \int_{\R_+^n \setminus B_R^+} \frac{ |x|^{n - a} }{ (|x' - y'|^2 + x_n^2)^\frac{n - a}{2} } \frac{ x_n^{1 - a} }{ |x|^{n - a} } \left[(\mathcal{P}_a \phi_i) (x)\right]^\frac{n - a + 2}{n + a - 2} d x \\
& \leq \frac{ R^{n - a} }{ (R - |y'|)^{n - a} } h_i (R, 0).
\endaligned
$$
It follows that
$$
\lim_{R \to \infty} h(R, y') = \lim_{R \to \infty} h(R, 0) = : c_2 \geq 0.
$$
Sending $R$ to $\infty$ in \eqref{eq:phtwoparts}, by the Lebesgue's monotone convergence theorem we have
$$
K \phi (y')^\frac{n - a}{n + a - 2} = \int_{\R_+^n} P_a (y', x) \left[(P_a \phi) (x)\right]^\frac{n - a + 2}{n + a - 2} d x + c_2.
$$
If $c_2 > 0$, then $\phi \geq \big( \frac{c_2}{c_0} \big)^\frac{n + a - 2}{n - a}$ and thus $P_a \phi \geq \big( \frac{c_2}{c_0} \big)^\frac{n + a - 2}{n - a}$. This is impossible, since otherwise the integral in the right-hand side is infinity. Hence $c_2 = 0$. The proof of Theorem \ref{thm:blowup} is completed.
\end{proof}

\section{A variational problem}\label{sec:variational}

Let $K \in C^1 (\partial B_1)$ be a positive function satisfying $K(\xi) = K(- \xi)$, and $L_{\rm as}^p (\partial B_1) \subset L^p (\partial B_1)$ ($p \geq 1$) be the set of antipodally symmetric functions. For $p \geq \frac{2 (n - 1)}{n + a - 2}$, define
$$
\lambda_{ {\rm as}, p } (K) = \sup \bigg\{ \int_{B_1} |\mathcal{P}_a v|^\frac{2 n}{n + a - 2} d \xi : v \in L_{\rm as}^p (\partial B_1) ~ \textmd{with} ~ \int_{\partial B_1} K |v|^p d s = 1 \bigg\}.
$$
Denote
$$
\lambda_{ {\rm as}, \frac{2 (n - 1)}{n + a - 2} } (K) = \lambda_{\rm as} (K).
$$

\begin{proposition}\label{prop:lambdaK} If
\begin{equation}\label{eq:lambdaKge}
\lambda_{\rm as} (K) > \frac{ \mathcal{S}_{n, a}^\frac{2 n}{n + a - 2} }{ (\min_{\partial B_1} K)^\frac{n}{n - 1} 2^\frac{1}{n - 1} },
\end{equation}
where $\mathcal{S}_{n, a}$ is the sharp constant in the inequality \eqref{Pa}, then $\lambda_{\rm as} (K)$ is achieved.
\end{proposition}

\begin{proof} We claim that
$$
\liminf_{ p \searrow \frac{2 (n - 1)}{n + a - 2} } \lambda_{ {\rm as}, p } (K) \geq \lambda_{\rm as} (K).
$$
For any $\varepsilon > 0$, by the definition of $\lambda_{\rm as} (K)$, we can find a function $v \in L_{\rm as}^\infty (\partial B_1)$ such that
$$
\int_{B_1} |\mathcal{P}_a v|^\frac{2 n}{n + a - 2} d \xi > \lambda_{\rm as} (K) - \varepsilon ~~~~~~ \textmd{and} ~~~~~ \int_{\partial B_1} K |v|^\frac{2 (n - 1)}{n + a - 2} d s = 1.
$$
Let $V_p : = \int_{\partial B_1} K |v|^p d s$. Since
$$
\lim_{ p \to \frac{2 (n - 1)}{n + a - 2} } V_p = \int_{\partial B_1} K |v|^\frac{2 (n - 1)}{n + a - 2} d s = 1, 
$$
we have,  for $p$ close to $\frac{2 (n - 1)}{n + a - 2}$ sufficiently, that 
$$
\lambda_{ {\rm as}, p } (K) \geq \int_{B_1} \bigg| \mathcal{P}_a \bigg( \frac{v}{ V_p^{1/p} } \bigg) \bigg|^\frac{2 n}{n + a - 2} d \xi \geq \lambda_{\rm as} (K) - 2 \varepsilon.
$$
Since $\varepsilon$ is arbitrary, the claim is proved.

By the above claim, we can find $p_i \searrow \frac{2 (n - 1)}{n + a - 2}$ as $i \to \infty$ such that $\lambda_{ {\rm as}, p_i } (K) \to \lambda \geq \lambda_{\rm as} (K)$. Since $K \in C^1 (\partial B_1)$ is positive, if follows from Corollary \ref{cor:compactembeddingS} that for $p_i > \frac{2 (n - 1)}{n + a - 2}$, $\lambda_{ {\rm as}, p_i }$ is achieved, say, by $v_i$. Since $|\mathcal{P}_a v_i| \leq \mathcal{P}_a |v_i|$, we can assume that $v_i$ is non-negative. Moreover,
$$
\|  v_i \|_{L^{p_i} (\partial B_1)}^{p_i} \leq \frac{1}{\min_{\partial B_1} K}.
$$
Then, by \eqref{eq:sharp} we have $\| \mathcal{P}_a v_i \|_{L^\frac{2 n}{n + a - 2} (B_1)} \leq C$ for some $C > 0$ independent of $i$. It is easy to see that $v_i$ satisfies the Euler-Lagrange equation
\begin{equation}\label{eq:equationviSp}
\lambda_{ {\rm as}, p_i } (K) K(\eta) v_i (\eta)^{p_i - 1} = \int_{B_1} P_a (\eta, \xi) \left[(\mathcal{P}_a v_i) (\xi)\right]^\frac{n - a + 2}{n + a - 2} d \xi ~~~~~ \forall ~ \eta \in \partial B_1.
\end{equation} 
By the regularity result in Theorem \ref{thm:Holder}, $v_i \in C^\beta (\partial B_1)$ for any $\beta \in (0, 1)$. 

Next we will show that $ v_i $ is uniformly bounded. Otherwise, we have  
$$
v_i (\eta_i) = \max_{\partial B_1} v_i \to \infty ~~~~~ \textmd{as} ~ i \to \infty.
$$ 
Let $\eta_i \to \bar{\eta}$ as $i \to \infty$. By the stereographic projection with $\eta_i$ as the south pole, equation \eqref{eq:equationviSp} is transformed to
$$
\lambda_{ {\rm as}, p_i } (K) K_i (y') u_i (y')^{p_i - 1} = \int_{\R_+^n} P_a (y', x) \left[(\mathcal{P}_a u_i) (x)\right]^\frac{n - a + 2}{n + a - 2} d x ~~~~~~ \forall ~ y' \in \R^{n - 1}, 
$$
where  
$$
K_i (y') = \bigg( \frac{ \sqrt{2} }{|y' + e_n|} \bigg)^{(n + a - 2) (p_i - 1) - n + a} K( F(y') )
$$
and 
$$ 
u_i (y') = \bigg( \frac{ \sqrt{2} }{|y' + e_n|} \bigg)^{n + a - 2} v_i ( F(y') ).
$$
Hence, $u_i (0) = \max_{ \R^{n - 1} } u_i \to \infty$ as $i \to \infty$. Taking $R_i = u_i(0)^{-\frac{1}{2} (p_i - \frac{2n}{n+a-2} ) } \to +\infty$ and using Theorem \ref{thm:blowup}, we obtain that after passing to a subsequence,  
$$
\phi_i (y') : = \frac{1}{u_i (0)} u_i \big( u_i (0)^{ p_i - \frac{2 n}{n + a - 2} } y' \big) \to \phi (y') ~~~~~~ \textmd{in} ~ C_{loc}^{1/2} (\R^{n - 1}),
$$
where $\phi > 0$ satisfies
$$
\lambda K(\bar{\eta}) \phi (y')^\frac{n - a}{n + a - 2} = \int_{\R_+^n} P_a (y', x) \left[(\mathcal{P}_a \phi) (x)\right]^\frac{n - a + 2}{n + a - 2} d x ~~~~~~ \textmd{for} ~ y' \in \R^{n - 1}.
$$
By Tang-Dou \cite{TD}, $\phi$ is classified.

Since $v_i$ is non-negative and antipodally symmetric, for any small $\delta > 0$ we have
$$
\aligned
1 & = \int_{\partial B_1} K v_i^{p_i} d s \\
& \geq 2 \int_{F(B_\delta')} K v_i^{p_i} d s \\
& = 2 \int_{B_\delta'} K_i u_i^{p_i} d z' \\
& = 2 u_i (0)^{n \big (p_i - \frac{2 (n - 1)}{n + a - 2} \big)} \int_{ B'_{ \delta u_i (0)^{\frac{2 n}{n + a - 2} - p_i} } } K_i \big( u_i (0)^{ p_i - \frac{2 n}{n + a - 2} } y' \big) \phi_i (y')^{p_i} d y' \\
& \geq 2 \int_{B_R'} K_i \big( u_i (0)^{ p_i - \frac{2 n}{n + a - 2} } y' \big) \phi_i (y')^{p_i} d y' \\
&\to 2 K(\bar{\eta}) \int_{B_R'} \phi (y')^\frac{2 (n - 1)}{n + a - 2} d y'
\endaligned
$$
as $i \to \infty$ for any fixed $R > 0$. It follows that
$$
1 \geq 2 K(\bar{\eta}) \int_{ \R^{n - 1} } \phi (y')^\frac{2 (n - 1)}{n + a - 2} d y'.
$$
Hence,
$$
\aligned
\mathcal{S}_{n, a}^\frac{2 n}{n + a - 2} & \geq \frac{ \int_{\R_+^n} |\mathcal{P}_a \phi|^\frac{2 n}{n + a - 2} }{ \big( \int_{ \R^{n - 1} } |\phi|^\frac{2 (n - 1)}{n + a - 2} \big)^\frac{n}{n - 1} } \\
& = \lambda K(\bar{\eta}) \bigg( \int_{ \R^{n - 1} } |\phi|^\frac{2 (n - 1)}{n + a - 2} \bigg)^{ - \frac{1}{n - 1} }\\ &\geq \lambda K(\bar{\eta})^\frac{n}{n - 1} 2^\frac{1}{n - 1}.
\endaligned
$$
It implies that
$$
\lambda \leq \frac{ \mathcal{S}_{n, a}^\frac{2 n}{n + a - 2} }{ (\min_{\partial B_1} K)^\frac{n}{n - 1} 2^\frac{1}{n - 1} },
$$
which contradicts the assumption \eqref{eq:lambdaKge}. Therefore, $\{ v_i \}$ is uniformly bounded on $\partial B_1$.   

By Theorem \ref{thm:Holder}, $\{ v_i \}$ is bounded in $C^{1/2} (\partial B_1)$. Thus, after passing to a subsequence, we have for some non-negative function $v \in C(\partial B_1)$, 
$$
\aligned
v_i  &\to v ~~~~~~ & \textmd{in}  ~ & ~ C(\partial B_1), \\ 
\endaligned
$$ 
and thus,
$$
\aligned
\mathcal{P}_a v_i  & \to  \mathcal{P}_a v & \textmd{in} ~ & ~  C(\overline B_1). 
\endaligned
$$ 
Letting $i \to \infty$ in \eqref{eq:equationviSp},  we obtain that $v$ satisfies
$$
\lambda K(\eta) v(\eta)^\frac{n - a}{n + a - 2} = \int_{B_1} P_a (\eta, \xi) \left[(\mathcal{P}_a v) (\xi)\right]^\frac{n - a + 2}{n + a - 2} d \xi.
$$
Moreover, since  
$$
1 =  \int_{\partial B_1} K(\eta) v_i (\eta)^{p_i} d\eta  \to  \int_{\partial B_1} K(\eta) v (\eta)^{\frac{2 (n - 1)}{n + a - 2}} d\eta, 
$$ 
we have $v > 0$ on $\partial B_1$. These also imply that $\lambda = \lambda_{\rm as}(K)$ and $\lambda_{\rm as} (K)$ is achieved. 
The proof of Proposition \ref{prop:lambdaK} is completed.
\end{proof}

\begin{proof}[Proof of Theorem \ref{thm:existence}] Let $v = 1$, then
$$
\lambda_{\rm as} (K) \geq \frac{ \int_{B_1} |\mathcal{P}_a 1|^\frac{2 n}{n + a - 2} d \xi }{ \big( \int_{\partial B_1} K d s \big)^\frac{n}{n - 1} } \geq \frac{ \mathcal{S}_{n, a}^\frac{2 n}{n + a - 2} }{ (\max_{\partial B_1} K)^\frac{n}{n - 1} } > \frac{ \mathcal{S}_{n, a}^\frac{2 n}{n + a - 2} }{ (\min_{\partial B_1} K)^\frac{n}{n - 1} 2^\frac{1}{n - 1} },
$$
where we use \eqref{eq:maxmin} in the last inequality. By Proposition \ref{prop:lambdaK}, we obtain the desired result.
\end{proof}

\begin{appendices}

\section{H\"older regularity}

This appendix is devoted to the proof of Theorem \ref{thm:Holder}. We start with the improvement of integrability of the subsolutions to some linear integral equations.
\begin{proposition}\label{prop:regularityone} Suppose that $n \geq 2$ and $2 - n < a < 1$. Let $1 < r, s \leq \infty$, $1 \leq t < \infty$, $\frac{n}{n - 1} < p < q < \infty$ satisfy
$$
\frac{1}{n} < \frac{t}{q} + \frac{1}{r} < \frac{t}{p} + \frac{1}{r} \leq 1
$$
and
$$
\frac{n}{t r} + \frac{n - 1}{s} = \frac{1}{t}.
$$
Assume that $U, V \in L^p (B_R^+)$, $W \in L^r (B_R^+)$, $f \in L^s (B_R')$ are all non-negative functions, $V \in L^q (B_{R/2}^+)$,
$$
\| W \|_{L^r (B_R^+)}^{1/t} \| f \|_{L^s (B_R')} \leq \varepsilon (n, a, p, q, r, s, t) ~ \mbox{small} 
$$
and
$$
U(x) \leq \int_{B_R'} P_a (y', x) f(y') \bigg( \int_{B_R^+} P_a (y', z) W(z) U(z)^t d z \bigg)^{1/t} d y' + V(x)
$$
for $x \in B_R^+$. Then $U \in L^q (B_{R/4}^+)$ and
$$
\| U \|_{ L^q (B_{R/4}^+) } \leq c(n, a, p, q, r, s, t)\Big( R^{ \frac{n}{q} - \frac{n}{p} } \| U \|_{ L^p (B_R^+) } + \| V \|_{ L^q (B_{R/2}^+) } \Big).
$$
\end{proposition}

The proof of Proposition \ref{prop:regularityone} is the same as that of \cite[Proposition 5.2]{HWY08}. We also need the following two $L^p$-boundedness for the operator $\mathcal{P}_a$ and its adjoint operator.
\begin{proposition}[Chen \cite{Chen14}]\label{prop:sharpwithp} Suppose that $n \geq 2$ and $2 - n < a < 1$. For $1 < p \leq \infty$ we have
$$
\| \mathcal{P}_a f \|_{ L^\frac{n p}{n - 1} (\R_+^n) } \leq c(n, a, p) \| f \|_{ L^p (\R^{n - 1}) }
$$
for any $f \in L^p (\R^{n - 1})$.
\end{proposition}

For a function $F$ on $\R_+^n$, define
$$
(\mathcal{T}_a F) (y') = \int_{\R_+^n} P_a (y', x) F(x) d x.
$$
Then we have the following inequality by a duality argument. See also the similar proof in \cite[Proposition 2.3]{HWY08}. 
\begin{proposition}\label{prop:adjsharpwithp} Suppose that $n \geq 2$ and $2 - n < a < 1$. For $1 \leq p < n$ we have
$$
\| \mathcal{T}_a F \|_{ L^\frac{(n - 1) p}{n - p} (\R^{n - 1}) } \leq C(n, a, p) \| F \|_{ L^p (\R_+^n) }
$$
for any $F \in L^p (\R_+^n)$.
\end{proposition}

Next we give the details of the proof of Theorem \ref{thm:Holder}.

\begin{proof}[Proof of Theorem \ref{thm:Holder}] Let $\tilde{u}_0 (y') = K(y') u(y')^{p - 1}$ and $U_0 (x) = (\mathcal{P}_a u) (x)$. Then
$$
\tilde{u}_0 (y') = \int_{\R_+^n} P_a (y', x) U_0 (x)^\frac{n - a + 2}{n + a - 2} d x.
$$
Define
$$
U_R (x) = \int_{ \R^{n - 1} \setminus B_R' } P_a (y', x) u(y') d y',
$$
$$
\tilde{u}_R (y') = \int_{\R_+^n \setminus B_R^+} P_a (y', x) U_0 (x)^\frac{n - a + 2}{n + a - 2} d x.
$$
Since $u \in L_{loc}^p (\R^{n - 1})$, by Proposition \ref{prop:sharpwithp} we get $\int_{B_R'} P_a (z', \cdot) u(z') d z' \in L^\frac{n p}{n - 1} (\R_+^n)$. Notice that
$$
\frac{n p}{n - 1} \geq \frac{2 n}{n + a - 2}.
$$

\vskip0.1in

{\bf Step 1.} We claim that $U_0 \in L_{loc}^\frac{n p}{n - 1} (\overline{\R_+^n})$ and $U_R \in L^\frac{n p}{n - 1} (B_R^+) \cap L_{loc}^\infty (B_R^+ \cup B_R')$.

\vskip0.1in

Since $u \in L_{loc}^p (\R^{n - 1})$, we have $u < \infty$ a.e. on $\R^{n - 1}$. It implies that $U_0 < \infty$ a.e. on $\R_+^n$. Hence, there exists $x_0 \in B_R^+$ such that $U_0 (x_0) < \infty$. It follows that
$$
\int_{ \R^{n - 1} \setminus B_R' } \frac{u(z')}{ (|x_0' - z'|^2 + x_{0, n}^2)^\frac{n - a}{2} } d z' < \infty.
$$
Thus,
$$
\int_{ \R^{n - 1} \setminus B_R' } \frac{u(z')}{ |z'|^{n - a} } d z' < \infty.
$$
For $0 < \theta < 1$ and $x \in B_{\theta R}^+$, we have
$$
U_R (x) = \int_{ \R^{n - 1} \setminus B_R' } P_a (z', x) u(z') d z' \leq \frac{ c_{n, a} R^{1 - a} }{ (1 - \theta)^{n - a} } \int_{ \R^{n - 1} \setminus B_R' } \frac{ u(z') }{ |z'|^{n - a} } d z'.
$$
It follows that $U_R \in L_{loc}^\infty (B_R^+ \cup B_R')$. Since $\int_{B_R'} P_a (z', \cdot) u(z') d z' \in L^\frac{n p}{n - 1} (\R_+^n)$, we know that $U_0 \in L_{loc}^\frac{n p}{n - 1} (B_R^+ \cup B_R')$. Since $R>0$ is arbitrary, we deduce that $U_0 \in L_{loc}^\frac{n p}{n - 1} (\overline{\R_+^n})$ and hence $U_R \in L^\frac{n p}{n - 1} (B_R^+)$.

\vskip0.1in

{\bf Step 2.} We show that $\tilde{u}_R \in L^\frac{p}{p - 1} (B_R') \cap L_{loc}^\infty (B_R')$.

\vskip0.1in

Since $\tilde{u}_0 \in L_{loc}^\frac{p}{p - 1} (\R^{n - 1})$, we obtain $\tilde{u}_0 \in L^\frac{p}{p - 1} (B_R')$ and thus $\tilde{u}_R \in L^\frac{p}{p - 1} (B_R')$. Hence, we can find $y_0' \in B_R'$ such that $\tilde{u}_R (y_0') < \infty$. That is, 
$$
\int_{\R_+^n \setminus B_R^+} \frac{ z_n^{1 - a} }{ (|z' - y_0'|^2 + z_n^2)^\frac{n - a}{2} } U_0 (z)^\frac{n - a + 2}{n + a - 2} d z < \infty.
$$
Therefore,
$$
\int_{\R_+^n \setminus B_R^+} \frac{ z_n^{1 - a} }{ |z|^{n - a} } U_0 (z)^\frac{n - a + 2}{n + a - 2} d z < \infty.
$$
For $0 < \theta < 1$ and $y' \in B_{\theta R}'$, we have
$$
\tilde{u}_R (y') = \int_{\R_+^n \setminus B_R^+} P_a (y', z) U_0 (z)^\frac{n - a + 2}{n + a - 2} d z \leq \frac{ c_{n, a} }{ (1 - \theta)^{n - a} } \int_{\R_+^n \setminus B_R^+} \frac{ z_n^{1 - a} }{ |z|^{n - a} } U_0 (z)^\frac{n - a + 2}{n + a - 2} d z.
$$
This implies that $\tilde{u}_R \in L_{loc}^\infty (B_R')$.

\vskip0.1in 

{\bf Step 3.} We prove that $\tilde{u}_0 \in L_{loc}^\infty (\R^{n - 1})$ and $U_0 \in L_{loc}^\infty (\overline{\R_+^n})$.

\vskip0.1in

{\it Case 1: } $\frac{2 (n - 1)}{n + a - 2} < p < \infty$. This is the subcritical case, and we directly use the bootstrap method to prove the regularity. 

\vskip0.1in

From Proposition \ref{prop:sharpwithp} and Step 1, we know that $U_0^\frac{n - a + 2}{n + a - 2} \in L_{loc}^{q_0} (\overline{\R_+^n})$ with 
$$
q_0 := \frac{n p}{n - 1} \cdot \frac{n + a - 2}{n - a + 2} > \frac{2 n}{n + a - 2} \cdot \frac{n + a - 2}{n - a + 2} = \frac{2 n}{n - a + 2} > 1.
$$
If $q_0 \geq  n$, by Proposition \ref{prop:adjsharpwithp} we know that $\int_{B_R^+} P_a (\cdot, z) U_0 (z)^\frac{n - a + 2}{n + a - 2} d z \in L^r (\R^{n - 1})$ for any $1 \leq r < \infty$. This together with Step 2 implies that $\tilde{u}_0 \in L_{loc}^r (B_R')$ for any $1 \leq r < \infty$. Since $R$ is arbitrary, we obtain $\tilde{u}_0 \in L_{loc}^r (\R^{n - 1})$ for any $1 \leq r < \infty$. Moreover, by Proposition \ref{prop:sharpwithp} we have $\int_{B_R'} P_a (z', \cdot) u(z') d z' \in L^s (\R_+^n)$ for any $\frac{n}{n - 1} < s < \infty$. Combined with Step 1, we also have $U_0 \in L_{loc}^s (\overline{\R_+^n})$ for any $1 \leq s < \infty$. If $q_0 < n$, then Proposition \ref{prop:adjsharpwithp} yields $\int_{B_R^+} P_a (\cdot, z) U_0 (z)^\frac{n - a + 2}{n + a - 2} d z \in L^\frac{(n - 1) q_0}{n - q_0} (\R^{n - 1})$. Combined with Step 2, we have $\tilde{u}_0 \in L_{loc}^\frac{(n - 1) q_0}{n - q_0} (B_R')$ for any $R>0$.  Consequently, we deduce that $u \in L_{loc}^{p_1} (\R^{n - 1})$ with
$$
p_1 := (p - 1) \cdot \frac{(n - 1) q_0}{n - q_0} = p \cdot \frac{ (p - 1) \frac{n + a - 2}{n - a + 2} }{ 1 - \frac{p (n + a - 2)}{(n - 1) (n - a + 2)} } > p, 
$$
where the last inequality holds since $p > \frac{2 (n - 1)}{n + a - 2}$. From now on, we denote the constant 
$$
\gamma:=\frac{ (p - 1) \frac{n + a - 2}{n - a + 2} }{ 1 - \frac{p (n + a - 2)}{(n - 1) (n - a + 2)} }>1.
$$
We can see that the regularity of $u$ is boosted to $L^{p_1}_{loc}(\R^{n - 1})$ with $p_1 =p \cdot  \gamma$. 

Using Proposition \ref{prop:sharpwithp} and Step 1 again, we obtain $U_0^\frac{n - a + 2}{n + a - 2} \in L_{loc}^{q_1} (\overline{\R_+^n})$ with 
$$
q_1 := \frac{n p_1}{n - 1} \cdot \frac{n + a - 2}{n - a + 2} = q_0 \cdot \gamma > q_0.
$$
If $q_1 \geq n$, then we easily obtain $U_0 \in L_{loc}^s (\overline{\R_+^n})$ for any $1 \leq s < \infty$. If $q_1 < n$, by a similar argument as above we can obtain that $u \in L_{loc}^{p_2} (\R^{n - 1})$ with
$$
p_2 := (p-1) \cdot \frac{(n-1)q_1}{n-q_1}=
p_1 \cdot \frac{ (p - 1) \frac{n + a - 2}{n - a + 2} }{ 1 - \frac{p_1 (n + a - 2)}{(n - 1) (n - a + 2)} } > p_1 \cdot \gamma 
$$
due to $p_1 > p$. Hence, the regularity of $u$ is boosted to $L^{p_2}_{loc}(\R^{n - 1})$ with $p_2 > p_1 \cdot  \gamma$. By Proposition \ref{prop:sharpwithp} and Step 1 again, we obtain $U_0^\frac{n - a + 2}{n + a - 2} \in L_{loc}^{q_2} (\overline{\R_+^n})$ with 
$$
q_2 := \frac{n p_2}{n - 1} \cdot \frac{n + a - 2}{n - a + 2} > q_1\cdot \gamma.
$$

Repeating this process with finite many steps, we can boost $U_0$ to $L_{loc}^q (\overline{\R_+^n})$ for any $1 \leq q < \infty$. By H\"{o}lder inequality we get 
$$
\tilde{u}_0 (y') = \int_{ B_R^+} P_a (y', z) U_0 (z)^\frac{n - a + 2}{n + a - 2} d z + \tilde{u}_R (y') \leq c(n, a, q) \| U_0 \|_{ L^q (B_R^+) }^\frac{n - a + 2}{n + a - 2} + \tilde{u}_R (y')
$$
for some $q > \frac{n(n - a + 2)}{n + a - 2}$. This together with Step 2 implies that $\tilde{u}_0 \in L_{loc}^\infty (B_R')$. Since $R$ is arbitrary, we have $\tilde{u}_0 \in L_{loc}^\infty (\R^{n - 1})$ and hence $u \in L_{loc}^\infty (\R^{n - 1})$. Combined with Step 1, we see $U_0 \in L_{loc}^\infty (\overline{\R_+^n})$.

\vskip0.1in 

{\it Case 2:} $p = \frac{2 (n - 1)}{n + a - 2}$. For this critical case, the bootstrap method above does not work. We will use Proposition \ref{prop:regularityone} to establish the regularity.  

\vskip0.1in

In this case, we have $U_0 \in L_{loc}^\frac{2 n}{n + a - 2} (\overline{\R_+^n})$ and $U_R \in L^\frac{2 n}{n + a - 2} (B_R^+) \cap L_{loc}^\infty (B_R^+ \cup B_R')$. Since $a < 1$, we get $0< \frac{n + a - 2}{n - a} < 1$. Then,
$$
\tilde{u}_0 (y')^\frac{n + a - 2}{n - a} \leq \bigg( \int_{B_R^+} P_a (y', z) U_0 (z)^\frac{n - a + 2}{n + a - 2} d z \bigg)^\frac{n + a - 2}{n - a} + \tilde{u}_R (y')^\frac{n + a - 2}{n - a}.
$$
Hence,
\begingroup
\allowdisplaybreaks
\begin{align*}
U_0 (x) = & \int_{B_R'} P_a (y', x) u (y') d y' + U_R (x) \\
= & \int_{B_R'} P_a (y', x) K(y')^{- \frac{n + a - 2}{n - a} } \tilde{u}_0 (y')^\frac{n + a - 2}{n - a} d y' + U_R (x) \\
\leq & \int_{B_R'} P_a (y', x) K(y')^{- \frac{n + a - 2}{n - a} } \bigg( \int_{B_R^+} P_a (y', z) U_0 (z)^\frac{2}{n + a - 2} U_0 (z)^\frac{n - a}{n + a - 2} d z \bigg)^\frac{n + a - 2}{n - a} d y' \\
& + V_R (x), 
\end{align*}
\endgroup
where
$$
V_R (x) = \int_{B_R'} P_a (y', x) K(y')^{ - \frac{n + a - 2}{n - a} } \tilde{u}_R (y')^\frac{n + a - 2}{n - a} d y' + U_R (x).
$$
Since $\tilde{u}_R \in L^\frac{2 (n - 1)}{n - a} (B_R')$, we have $V_R \in L^\frac{2 n}{n + a - 2} (B_R^+)$. 
On the other hand, for $0 < \theta < 1$, $x \in B_{\theta R}^+$, we have
$$
\aligned
& \int_{B_R'} P_a (y', x) K(y')^{- \frac{n + a - 2}{n - a} } \tilde{u}_R (y')^\frac{n + a - 2}{n - a} d y' \\
& \leq (\min\nolimits_{B_R'} K)^{- \frac{n + a - 2}{n - a} } \bigg[ \| \tilde{u}_R \|_{ L^\infty (B_{ \frac{1 + \theta}{2} R }) }^\frac{n + a - 2}{n - a} + \frac{c(n, a)}{ (1 - \theta)^{n - a} R^{n - 1} } \int_{ B_R' \setminus B_{ \frac{1 + \theta}{2} R } } \tilde{u}_R (y')^\frac{n + a - 2}{n - a} d y' \bigg] \\
& \leq (\min\nolimits_{B_R'} K)^{- \frac{n + a - 2}{n - a} } \bigg[ \| \tilde{u}_R \|_{ L^\infty (B_{ \frac{1 + \theta}{2} R }) }^\frac{n + a - 2}{n - a} + \frac{c(n, a)}{ (1 - \theta)^{n - a} R^\frac{n + a - 2}{2} } \| \tilde{u}_R \|_{ L^\frac{2 (n - 1)}{n - a} (B_R') }^\frac{n + a - 2}{n - a} \bigg].
\endaligned
$$
Hence, $V_R \in L_{loc}^\infty (B_R^+ \cup B_R')$. It follows from Proposition \ref{prop:regularityone} that $U_0 \in L^q (B_{R/4}^+)$ for any $\frac{2 n}{n + a - 2} < q < \infty$ when $R$ is sufficiently small. Therefore,
$$
\tilde{u}_0 (y') = \int_{ B_{R/4}^+ } P_a (y', z) U_0 (z)^\frac{n - a + 2}{n + a - 2} d z + \tilde{u}_{R/4} (y') \leq c(n, a, q) \| U_0 \|_{ L^q (B_{R/4}^+) }^\frac{n - a + 2}{n + a - 2} + \tilde{u}_{R/4} (x')
$$
for some $q > \frac{n(n - a + 2)}{n + a - 2}$. In particular, we see $\tilde{u}_0 \in L^\infty (B_{R/8}')$. Since every point can be viewed as a center, we get $\tilde{u}_0 \in L_{loc}^\infty (\R^{n - 1})$ and hence $U_0 \in L_{loc}^\infty (\overline{\R_+^n})$.

\vskip0.1in

{\bf Step 4.} We prove that $u \in C_{loc}^\beta (\R^{n - 1})$ for any $\beta \in (0, 1)$.

\vskip0.1in

From Step 1 and 2, we know that for any $R > 0$,
$$
\int_{ \R^{n - 1} \setminus B_R' } \frac{u(y')}{ |y'|^{n - a} } d y' < \infty ~~~~~~ \textmd{and} ~~~~~~ \int_{\R_+^n \setminus B_R^+} \frac{ x_n^{1 - a} }{ |x|^{n - a} } U_0 (x)^\frac{n - a + 2}{n + a - 2} d x < \infty.
$$
Therefore, $\tilde{u}_R \in C^\infty (B_R')$ and $U_R \in C^{1-a}(B_R^+ \cup B_R')$. It follows from Step 3 that $\tilde{u}_0 \in C_{loc}^\beta (\R^{n - 1})$ for any $0 < \beta < 1$. By the continuity, $\tilde{u}_0 > 0$ in $\R^{n - 1}$. Consequently, $u \in C_{loc}^\beta (\R^{n - 1})$ for any $0 < \beta < 1$ since $K$ is a positive $C^1$ function in $\R^n$.
\end{proof}

\end{appendices}

\bigskip

\noindent X. Du

\noindent School of Mathematical Sciences, Laboratory of Mathematics and Complex Systems, MOE, Beijing Normal University\\
Beijing 100875, China\\[1mm]
Email: \textsf{xduah@bnu.edu.cn}

\medskip

\noindent T. Jin

\noindent Department of Mathematics, The Hong Kong University of Science and Technology\\
Clear Water Bay, Kowloon, Hong Kong\\[1mm]
Email:  \textsf{tianlingjin@ust.hk}

\medskip

\noindent H. Yang

\noindent School of Mathematical Sciences, Shanghai Jiao Tong University\\
Shanghai 200240, China\\[1mm]
Email: \textsf{hui-yang@sjtu.edu.cn}

\end{document}